\documentclass[11pt, a4paper]{article}
\usepackage[T1]{fontenc}
\usepackage{amsmath, amssymb, amsthm}
\usepackage[applemac]{inputenc}
\usepackage{bm}
\usepackage{bbm}
\usepackage[colorlinks=true, linkcolor=blue, citecolor = blue]{hyperref}
\usepackage{natbib}
\usepackage{dsfont}
\usepackage{graphicx}

\bibpunct{(}{)}{,}{a}{,}{,}

\usepackage[top=2.5cm, bottom=2.5cm, left=2.25cm, right=2.25cm]{geometry}

   \newtheoremstyle{example}{\topsep}{\topsep}%
     {}
     {}
     {\bfseries}
     {.}
     {  }
     {}

\theoremstyle{plain}
\newtheorem{thm}{Theorem}
\newtheorem*{thm*}{Theorem}
\newtheorem{lem}{Lemma}

\newtheorem{prop}{Proposition}
\theoremstyle{definition}
\newtheorem*{defn}{Definition}
\theoremstyle{remark}

\newcommand{\N}{\mathbb{N}}
\renewcommand{\P}{\mathbb{P}}

\renewcommand{\L}{\mathcal{L}}

\newcommand{\R}{\mathbb{R}}

\newcommand{\F}{\mathcal{F}}
\newcommand{\B}{\mathcal{B}}
\renewcommand{\H}{\mathcal{H}}

    \def\independenT#1#2{\mathrel{\setbox0\hbox{$#1#2$}
    \copy0\kern-\wd0\mkern4mu\box0}}

\title{Large deviations for Hilbert space valued Wiener processes: a sequence space approach}
\author{Andreas Andresen, Peter Imkeller, Nicolas Perkowski\\ Institut f\"ur Mathematik\\ Humboldt-Universit\"at zu Berlin\\
Rudower Chaussee 25\\ 12489 Berlin\\ Germany}

\begin{document}

\maketitle
\begin{center}
   \textit{Dedicated to David Nualart on the Occasion of his 60th Birthday}
\end{center}

\begin{abstract}
Ciesielski's isomorphism between the space of $\alpha$-H\"older continuous
functions and the space of bounded sequences is used to give an alternative
proof of the large deviation principle for Wiener processes with values in  Hilbert space.
\end{abstract}

\noindent \emph{Mathematical Subjects Classification 2010:} 60F10; 60G15.


\section*{Introduction}

The large deviation principle (LDP) for Brownian motion $\beta$ on $[0,1]$
- contained in Schilder's theorem (\cite{Schilder}) - describes the exponential decay of the
probabilities with which $\sqrt{\varepsilon} \beta$ takes values in closed
or open subsets of the path space of continuous functions in which the
trajectories of $\beta$ live. The path space is equipped with the topology
generated by the uniform norm. The decay is dominated by a rate function
capturing the 'energy' $\frac{1}{2} \int_0^1 (\dot{f}(t))^2 dt$ of
functions $f$ on the Cameron-Martin space for which a square integrable
derivative exists. Schilder's theorem is of central importance to the
theory of large deviations for randomly perturbed dynamical systems or
diffusions taking their values in spaces of continuous functions (see
\cite{Freidlin}, \cite{Dembo}, and references therein, \cite{Vares}).  A
version of Schilder's theorem for a $Q$-Wiener processes $W$ taking values
in a separable Hilbert space $H$ is well known (see \cite{DaPrato}, 
Theorem 12.7 gives an LDP for Gaussian laws on Banach spaces). Here
$Q$ is a self adjoint positive trace class operator on $H$. If
$(\lambda_i)_{i\ge 0}$ are its summable eigenvalues with respect to an
eigenbasis $(e_k)_{k\ge 0}$ in $H$, $W$ may be represented with respect to
a sequence of one dimensional Wiener processes $(\beta_k)_{k\ge 0}$ by $W =
\sum_{k=0}^\infty \lambda_k \beta_k\, e_k$. The LDP in this framework can
be derived by means of techniques of reproducing kernel Hilbert spaces (see
\cite{DaPrato}, Chapter 12.1). The rate 
function is then given by an analogous energy functional for which $\dot{f}^2$
is replaced by $\lVert Q^{-\frac{1}{2}}\dot{F}\rVert^2$ for continuous functions
$F$ possessing square integrable derivatives $\dot{F}$ on $[0,1]$.

Schilder's theorem for $\beta$ may for instance be derived via
approximation of $\beta$ by random walks from LDP principles for discrete
processes (see \cite{Dembo}). \cite{Baldi} give a very elegant alternative
proof of Schilder's theorem, the starting point of which is a Fourier
decomposition of $\beta$ by a complete orthonormal system (CONS) in
$L^2([0,1])$. The rate function for $\beta$ is then simply calculated by
the rate functions of one-dimensional Gaussian unit variables. In this
approach, the LDP is first proved for balls of the topology, and then
generalized by means of exponential tightness to open and closed sets of
the topology. As a special feature of the approach, Schilder's theorem is
obtained in a stricter sense on all spaces of H\"older continuous functions
of order $\alpha<\frac{1}{2}$. This enhancement results quite naturally
from a characterization of the H\"older topologies on function spaces by
appropriate infinite sequence spaces (see \cite{Ciesielski}). Representing
the one-dimensional Brownian motions $\beta_k$ for instance by the CONS of
Haar functions on $[0,1]$, we obtain a description of the Hilbert space
valued Wiener process $W$ in which a double sequence of independent
standard normal variables describes randomness. Starting with this
observation, in this paper we extend the direct proof of Schilder's theorem
by \cite{Baldi} to $Q$-Wiener processes $W$ with values on
$H$. On the way, we also retrieve the enhancement of the LDP to spaces of
H\"older continuous functions on $[0,1]$ of order $\alpha<\frac{1}{2}$. The
idea of approaching problems related to stochastic processes with values in
function spaces by sequence space methods via Ciesielski's isomorphism is
not new: it has been employed in \cite{BenarousGradinaru} to give an
alternative treatment of the support theorem for Brownian motion, in
\cite{BenarousLedoux} to enhance the Freidlin-Wentzell theory from the
uniform to H\"older norms, and in \cite{Eddahbi} and \cite{EddahbiOuknine}
further to Besov-Orlicz spaces.

In Section \ref{preliminaries} we first give a generalization of
Ciesielski's isomorphism of spaces of H\"older continuous functions and
sequence spaces to functions with values on Hilbert spaces. We briefly
recall the basic notions of Gaussian measures and Wiener processes on
Hilbert spaces. Using Ciesielski's isomorphism we give a Schauder
representation of Wiener processes with values in $H$. Additionally we give
a short overview of concepts and results from the theory of LDP needed in
the derivation of Schilder's theorem for $W$. In the main Section \ref{LDP}
the alternative proof of the LDP for $W$ is given. We first introduce a new
norm on the space of H\"older continuous functions $C_\alpha([0,1],H)$ with
values in $H$ which is motivated by the sequence space representation in
Ciesielski's isomorphism, and generates a coarser topology. We adapt the
description of the rate function to the Schauder series setting, and then
prove the LDP for a basis of the coarser topology using Ciesielski's
isomorphism. We finally establish the last ingredient, the crucial property
of exponential tightness, by construction of appropriate compact sets in
sequence space.

\section{Preliminaries}\label{preliminaries}

In this section we collect some ingredients needed for the proof of
a large deviations principle for Hilbert space valued Wiener processes.
We first prove Ciesielski's theorem for Hilbert
space valued functions which translates properties of functions into properties of
the sequences of their Fourier coefficients with respect to complete orthonormal systems
in $L^2([0,1])$. We summarize some basic properties of Wiener
processes $W$ with values in a separable Hilbert space $H$. We then discuss
Fourier decompositions of $W$, prove that its trajectories
lie almost surely in $C_\alpha^0([0,1],H)$ and describe its
image under the Ciesielski isomorphism. We will always denote by $H$ a
separable Hilbert space equipped with a symmetric inner product
$\langle\cdot, \cdot\rangle$ that induces the norm $\lVert \cdot\rVert_H$ and a
countable complete orthonormal system (CONS) $(e_k)$ $k\in \mathbb{N}$.

\subsection{Ciesielski's isomorphism}

The \textbf{Haar functions} $(\chi_n, n \ge 0)$ are defined as $\chi_0 \equiv 1$,
\begin{align} \label{eq:haar functions def}
   \chi_{2^k+l}(t) :=  \begin{cases}
                                     \sqrt{2^k}, & \frac{2l}{2^{k+1}} \le t < \frac{2l+1}{2^{k+1}},\\
                                     -\sqrt{2^k}, & \frac{2l+1}{2^{k+1}} \le t \le \frac{2l+2}{2^{k+1}}, \\
                                     0, & \text{otherwise.}
                                  \end{cases}
\end{align}
The Haar functions form a CONS of $L^2([0,1],dx)$. Note that because of their wavelet structure, the integral $\int_{[0,1]} \chi_n
df$ is well-defined for all functions $f$. For $n=2^k+l$ where $k\in \N$
and $0\le l\le 2^k-1$ we have $\int_{[0,1]} \chi_n
dF=\sqrt{2^k}[2F(\frac{2l+1}{2^{k+1}})-F(\frac{2l+2}{2^{k+1}})-
F(\frac{2l}{2^{k+1}})]$, and it does not matter whether $F$ is a real or
Hilbert space valued function.

The primitives of the Haar functions are called \textbf{Schauder functions}, and they are given by
   \begin{align*}
      \phi_n (t) = \int_0^t \chi_n(s) ds\text{, }t\in[0,1],\text{ }n\ge 0.
   \end{align*}

Slightly abusing notation, we denote the $\alpha$-H\"older seminorms on
$C_\alpha([0,1];H)$ and on $C_\alpha([0,1];\R)$ by the same symbols
\begin{align*}
   \lVert F\rVert_\alpha &:= \sup_{0 \le s < t \le 1} \frac{ \lVert F(t) - F(s)\rVert_H}{ |t-s|^\alpha}, \quad F \in C_\alpha([0,1];H), \\
   \lVert f\rVert_\alpha &:= \sup_{0 \le s < t \le 1} \frac{ |f(t) - f(s)|}{ |t-s|^\alpha}, \quad f \in C_\alpha([0,1];\R)
\end{align*}
$C_\alpha([0,1];H)$ is of course the space of all functions $F: [0,1] \rightarrow H$ such that $\lVert F\rVert_\alpha < \infty$, and similarly for $C_\alpha([0,1];\R)$. We also denote the supremum norm on $C([0,1]; H)$ and $C([0,1]; \R)$ by the same symbol $\lVert \cdot \rVert_\infty$.

Denote in the sequel for an $H$-valued
function $F$ its orthogonal component with respect to $e_k$ by $F_k =
\langle F, e_k\rangle, k\ge 0.$ Further denote by $P_k$ (resp. $R_k$) the
orthogonal projectors on $\mbox{span}(e_1,\cdots, e_k)$ (resp. its
orthogonal complement), $k\ge 0.$ For every $F \in C_\alpha([0,1];H)$, every
$k\ge 0, s,t\in[0,1]$ we have
$$
 | \langle F(t), e_k \rangle - \langle F(s), e_k \rangle |
    \le \lVert F(t) - F(s)\rVert_H
$$
More generally, for any $k\ge 0, s,t\in[0,1]$ we have
$$ \lVert P_k F(t) - P_k F(s)\rVert_H
    \le
 \lVert F(t) - F(s)\rVert_H, \quad \lVert R_k F(t) - R_k F(s)\rVert_H
    \le
    \lVert F(t) - F(s)\rVert_H.
$$

Our approach starts with the observation that we may decompose functions
$F\in C_\alpha ([0,1];H)$ by double series with respect to the system
$(\phi_n\,e_k: n,k\ge 0)$.

\begin{lem}\label{lem:hoelderschauder}
 Let $\alpha\in (0,1)$ and $F\in C_\alpha([0,1];H)$. Then we have
\begin{align*}
 F=\sum_n \int_{[0,1]} \chi_n dF \phi_n = \sum_{n=0}^\infty \sum_{k=0}^\infty \int_{[0,1]} \chi_n dF_k e_k\phi_n
\end{align*}
with convergence in the uniform norm on $C([0,1];H)$.
\end{lem}
\begin{proof}
For the real valued functions $F_k, k\ge 0,$ the representation
$$F_k = \sum_{n=0}^\infty \int_{[0,1]} \chi_n d F_k\, \phi_n$$
is well known from \citet{Ciesielski}. Therefore we may write for $F\in
C_\alpha([0,1];H)$
\begin{align*}
F&=\sum_{k=0}^\infty F_ke_k\\
&=\sum_{k=0}^\infty e_k\sum_{n=0}^\infty \int_{[0,1]} \chi_n dF_k \phi_n \\
&=\sum_{n=0}^\infty \sum_{k=0}^\infty \int_{[0,1]} \chi_n dF_k e_k\phi_n \\
&=\sum_{n=0}^\infty \int_{[0,1]} \chi_n dF \phi_n.
\end{align*}
To justify the exchange in the order of summation and the convergence in
the uniform norm, we have to show
$$\lim_{N,m\to \infty} \left\rVert\sum_{n\ge N} \int_{[0,1]} \chi_n d R_m F
\phi_n\right\rVert_\infty = 0.$$ For this purpose, note first that by definition of
the Haar system for any $n,m\ge 0, n=2^k + l$, where $0\le l\le 2^k-1$
\begin{align*}
 \left\lVert \int_{[0,1]} \chi_n dR_m F\right\rVert_H&=\sqrt{2^k}\left\lVert 2R_m F\left(\frac{2l+1}{2^{k+1}}\right)-R_m F\left(\frac{2l+2}{2^{k+1}}\right)-
 R_m F\left(\frac{2l}{2^{k+1}}\right)\right\rVert_H\\
&\leq 2\lVert R_m F\rVert_\alpha 2^{-\alpha(k+1)}2^{\frac{1}{2}k}\\
&=\lVert R_m F\rVert_\alpha 2^{-\alpha(k+1)+\frac{1}{2}k+1}.
\end{align*}
Therefore and for $K\ge 0$ such that $2^K\le N\le 2^{K+1}$, using the fact
that $\phi_{2^k+l}, 0\le l\le 2^k-1$ have disjoint support and that
$\lVert\phi_{2^k+l}\rVert_\infty \le 2^{-\frac{k}{2}-1}$, we obtain
\begin{align*}
 \left\lVert \sum_{n\geq N} \int_{[0,1]}\chi_nd R_m F\phi_n\right\rVert_\infty&\leq \sum_{k\geq K}\left\lVert \sum_{0\leq l\leq 2^k-1} \int_{[0,1]}\chi_{2^k+l}dF\phi_{2^k+l}\right\rVert_\infty\\
&\leq \sum_{k\geq K} \sup_{0\leq l\leq 2^k-1} \left\lVert \int_{[0,1]} \chi_{2^k+l}d R_m F \right\rVert_\infty2^{-\frac{k}{2}-1}\\
&\leq \sum_{k\geq K} \lVert R_m F \rVert_\alpha 2^{-\alpha (k+1)}\\
&\le\lVert R_m F \rVert_\alpha \sum_{k\geq K}
(2^\alpha)^{-k}\xrightarrow[K,m\rightarrow \infty]{} 0.
\end{align*}
Here we use $\lVert R_m F\rVert_\alpha \le \lVert F\rVert_\alpha<\infty$ for all $m\ge 0$,
the fact that $\lim_{m\to\infty} R_m F(t)=0$ for any $t\in[0,1]$, and
dominated convergence to obtain $\lim_{m\to\infty} \lVert R_m F\rVert_\alpha = 0$.
\end{proof}
A closer inspection of the coefficients in the decomposition of Lemma
\ref{lem:hoelderschauder} leads us to the following isomorphism, described
by \citet{Ciesielski} in the 1-dimensional case. To formulate it, denote by
$\mathcal{C}_0^H$ the space of $H$-valued sequences $(\eta_n)_{n \in \N}$ such that $\lim_{n\rightarrow \infty}\lVert\eta_n\rVert_H=0$. If we equip $\mathcal{C}_0^H$ with the supremum norm (using again the symbol $\lVert\cdot \rVert_\infty$), it becomes a Banach space.

\begin{thm}[Ciesielski's isomorphism for Hilbert spaces]\label{CiesilskiH}
   Let $0 <  \alpha <  1$. Let $(\chi_n)$ denote the Haar functions, and $(\phi_n)$ denote the Schauder functions. Let for $0\le n=2^k+l\ge0$, where $0\le l\le 2^k-1$
   \begin{align*}
      c_0(\alpha) := 1, \quad c_n(\alpha):=2^{k(\alpha-1/2) + \alpha - 1}.
   \end{align*}
   Define
   \begin{align*}
       T^H_\alpha: C_\alpha^0([0,1];H) \rightarrow \mathcal{C}_0^H \qquad F \mapsto \left(c_n(\alpha) \int_{[0,1]} \chi_n dF\right)_{n \in \N}
   \end{align*}
   Then $T^H_\alpha$ is continuous and bijective,
   its operator norm is 1, and its inverse is given by
   \begin{align*}
       (T^H_\alpha)^{-1}: \mathcal{C}_0^H \rightarrow C_\alpha^0([0,1];H), \quad (\eta_n) \mapsto \sum_{n=0}^\infty \frac{\eta_n}{c_n(\alpha)} \phi_n,
   \end{align*}
   The norm of $(T^H_\alpha)^{-1}$ is bounded by
   \begin{align*}
       \left\lVert(T^H_\alpha)^{-1}\right\rVert \le \frac{2}{(2^\alpha-1)(2^{1-\alpha}-1)}.
   \end{align*}
 \end{thm}
\begin{proof}
Observe that for $n\in \N$ with $n=2^{k}+l$, $0\le l\le 2^{k}-1$
\begin{align*}
 &\left\lVert \int_{[0,1]} \chi_n dF\right\rVert_H\\
&=\sqrt{2^{k}}\left\lVert 2F\left(\frac{2l+1}{2^{k+1}}\right)-F\left(\frac{2l+2}{2^{k+1}}\right)-F\left(\frac{2l}{2^{k+1}}\right)\right\rVert_H\\
&\leq \frac{1}{2c_\alpha(n)}\left(\frac{\left\lVert F(\frac{2l+2}{2^{k+1}})-F(\frac{2l+1}{2^{k+1}})
\right\rVert_H}{2^{-\alpha(k+1)}}+\frac{\left\lVert F(\frac{2l+1}{2^{k+1}})-F(\frac{2l}{2^{k+1}})
\right\rVert_H}{2^{-\alpha(k+1)}}\right)\\
&\le  \frac{1}{c_\alpha(n)}
\sup_{t,s\in[0,1],\,\,|t-s|\le2^{-k-1}}\frac{\lVert
F(t)-F(s)\rVert_H}{|t-s|^\alpha}\\ &\leq \frac{1}{c_\alpha(n)}\lVert
F\rVert_\alpha.
\end{align*}
This gives the desired bound on the norm. Moreover, since $F\in
C_\alpha^0([0,1],H)$ we have
\begin{align*}
 \lim_{n\rightarrow \infty} c_\alpha(n)\left\lVert \int_{[0,1]} \chi_n dF\right\rVert_H\le\lim_{n\rightarrow \infty}
 \sup_{t,s\in[0,1],\,\,|t-s|\le2^{-k-1}}\frac{\lVert F(t)-F(s)\rVert_H}{|t-s|^\alpha}=0.
\end{align*}
Thus the range of $T^H_\alpha$ is indeed contained in $\mathcal{C}_0^H$.
Taking $F:[0,1]\rightarrow H$ with $F(s)=se_1$ for $s\in[0,1]$ we find that
$T^H_\alpha(F)=(e_1,0,0,...)$, thus $ \lVert F\rVert_\alpha= \lVert
T^H_\alpha (F)\rVert_\infty$. Therefore $\lVert T^H_\alpha \rVert=1$.
Clearly $T^H_\alpha$ is injective.

To see that $T^H_\alpha$ is bijective and that the inverse is bounded as claimed,
define
\begin{align*}
 A: \mathcal{C}_0^H \rightarrow C_\alpha^0([0,1];H), \quad (\eta_n) \mapsto \sum_{n=0}^\infty
 \frac{\eta_n}{c_n(\alpha)} \phi_n.
\end{align*}
Now a straightforward calculation using the orthogonality of the
$(\chi_n)_{n\ge 0}$ gives for any $(\eta_n)_{n\ge 0}\subset \mathcal{C}^H_0$
\begin{align*}
 T^H_\alpha\circ A((\eta_n)_{n\ge 0})&=T^H_\alpha\left(\sum_{n=0}^\infty \frac{\eta_n}{c_n(\alpha)}\phi_n\right)\\
&=\left(\sum_{n,m=0}^\infty\eta_n\int_{[0,1]}\chi_md\phi_n\right)_{m\in\N}\\
&=\left(\sum_{n,m=0}^\infty \eta_n\int\chi_n(t)\chi_m(t)dt\right)_{m\in\N}\\
&=(\eta_m)_{m\ge 0}.
\end{align*}
Consequently we can infer that $A=(T^H_\alpha)^{-1}$.\\

We still have to show that $(T^H_\alpha)^{-1}$ satisfies the claimed norm
inequality and maps every sequence $(\eta_n)_{n\ge 0}\in \mathcal{C}_0^H$
to an element of $C^0_\alpha([0,1],H)$. For this purpose let
$(\eta_n)_{n\ge 0}\in \mathcal{C}_0^H$, set
$F=(T^H_{\alpha})^{-1}((\eta_n))$ and let $s,t\in[0,1]$ be given. Then we have
\begin{align*}
 \lVert F(t)-F(s)\rVert_H\leq \lVert (\eta_n)_{n\ge 0}\rVert_{\infty}\left(\lvert t-s\rvert +\sum_{k=0}^{\infty}\sum_{l=0}^{2^k-1}\frac{\lvert \phi_{2^k+l}(t)-\phi_{2^k+l}(s)\rvert}{c_{2^k}(\alpha)}\right).\\
\end{align*}
The term in brackets on the right hand side is exactly the one appearing in
the real valued case (\citet{Ciesielski}). Consequently we have the same
bound, given by
\begin{align*}
 \lVert (T^H_{\alpha})^{-1}\rVert\leq \frac{1}{(2^\alpha-1)(2^{\alpha-1}-1)}.
\end{align*}
A more careful estimation yields
\begin{align*}
 \lVert F(t)-F(s)\rVert_H\leq \lVert \eta_0\rVert \lvert t-s \rvert + \sum_{k=0}^{\infty}\sum_{l=0}^{2^k-1}\frac{1}{c_{2^k}(\alpha)}\lVert \eta_{2^k+l}\rVert \lvert \phi_{2^k+l}(t)-\phi_{2^k+l}(s)\rvert.\\
\end{align*}
This is the same expression as in the real valued case. Its well known
treatment implies
\begin{align*}
 \lim_{\lvert t-s \rvert\rightarrow 0}\frac{\lVert F(t)-F(s)\rVert_H}{\lvert t-s\rvert^\alpha }=0.
\end{align*}
This finishes the proof.
\end{proof}

\subsection{Wiener processes on Hilbert spaces}
We recall some basic concepts of Gaussian random variables and Wiener processes with values in
a separable Hilbert space $H$. Especially we will derive a Fourier sequence decomposition of Wiener processes.
Our presentation follows \citet{DaPrato}.

\begin{defn}Let $(\Omega, \F, \mathbb{P})$ be a probability space, $m\in H$ and $Q:H\rightarrow H$
a positive self adjoint operator. An $H$-valued random variable $X$ such that for every $h\in H$
\begin{align*}
   E[\exp(i\langle h, X\rangle)]=\exp\left(i\langle h, m\rangle-\frac{1}{2}\langle Qh,h\rangle\right).
\end{align*}
is called Gaussian with covariance operator $Q$ and mean $m\in H$.
We denote the law of $X$ by $\mathcal{N}(m, Q)$.
\end{defn}
By Proposition 2.15 of \citet{DaPrato}, $Q$ has to be a positive, self-adjoint trace class
operator, i.e. a bounded operator from $H$ to $H$ that satisfies
\begin{enumerate}
   \item $\langle  Qx,x \rangle \ge 0 \text{ for every } x \in H$,
   \item $\langle  Qx,x \rangle = \langle  x, Qx\rangle \text{ for every } x \in H$,
   \item $\sum_{k=0}^\infty \langle  Q e_k, e_k \rangle <  \infty$ for every CONS $(e_k)_{k\ge 0}$.
\end{enumerate}
If $Q$ is a positive, self-adjoint trace class operator on $H$, then there exists a CONS $(e_k)_{k\ge 0}$ such that $Q e_k = \lambda_k
e_k$, where $\lambda_k \ge 0$ for all $k$ and $\sum_{k=0}^\infty \lambda_k <
\infty$. Note that for such a $Q$, an operator $Q^{1/2}$ can be defined by
setting $Q^{1/2} e_k := \sqrt{\lambda_k} e_k, k \in \N_0$. Then $Q^{1/2}
Q^{1/2} = Q$.

\begin{defn}Let $Q$ be a positive, self-adjoint trace class operator on $H$. A $Q$-Wiener process $(W(t): t \in [0,1])$ is a stochastic process with values in $H$ such that
\begin{enumerate}
   \item $W(0) = 0$,
   \item $W$ has continuous trajectories,
   \item $W$ has independent increments,
   \item $\L(W(t) - W(s)) = \mathcal{N}(0, (t-s) Q)$
\end{enumerate}
\end{defn}
In this case $(W(t_1), \dots , W(t_n))$ is $H^n$-valued Gaussian for all
$t_1, \dots, t_n \in [0,1]$. By Proposition 4.2 of \citet{DaPrato} we know
that such a process exists for every positive, self-adjoint trace class
operator $Q$ on H. To get the Fourier decomposition of a $Q$-Wiener process
along the Schauder basis we use a different standard characterization.
\begin{lem}\label{lem:characterisation of Wienerprocess}
 A stochastic process $Z$ on $(H, \B(H))$ is a $Q$-Wiener process if and only if
\begin{itemize}
 \item $Z_0=0$ $\mathbb{P}$-a.s.,
 \item $Z$ has continuous trajectories,
 \item $cov(\langle v, Z_t\rangle \langle w, Z_s\rangle)=(t\wedge s) \langle v,Qw \rangle$ $\forall v,w\in H$, $\forall 0\leq s\leq t<\infty$,
 \item $\forall (v_1,...,v_n)\in H^n$ $(\langle v_1,Z  \rangle,...,\langle v_n,Z\rangle)$
 is a $\R^n$-valued Gaussian process.
\end{itemize}
\end{lem}

Independent Gaussian random variables with values in a Hilbert space
asymptotically allow the following bounds.
\begin{lem}\label{lem:estimate for Gaussian RV}
 Let $Z_n\sim \mathcal{N}(0,Q)$, $n\in \mathbb{N}$, be independent. Then there exists an a.s. finite real valued random
 variable C such that
\begin{align*}
 \lVert Z_n\rVert_H\leq C\sqrt{\log n}\text{ }\mathbb{P}\text{ }a.s..
\end{align*}
\end{lem}
\begin{proof}
By using the exponential integrability of $\lambda \lVert Z_n\rVert^2_H$ for small
enough $\lambda$ and Markov's inequality, we obtain that there exist
$\lambda, c\in \R_+$ such that for any $a>0$
\begin{align*}
 \P(\lVert Z\rVert_H>a)\le ce^{-\lambda a^2}.
\end{align*}
Thus for $\alpha > 1$ and $n$ big enough
\begin{align*}
 \mathbb{P}\left(\lVert Z_n\rVert_H\geq \sqrt{\lambda^{-1}\alpha \log{n}}\right)\leq c n^{-\alpha}.
\end{align*}
We set $A_n=\left\{ \lVert Z_n\rVert_H\geq \sqrt{\lambda^{-1}\alpha \log{n}}\right\}$ and have
\begin{align*}
 \sum_{n=0}^{\infty}\mathbb{P}(A_n)< \infty.
\end{align*}
Hence the lemma of Borel-Cantelli gives, that $\mathbb{P}(\limsup_n
A_n)=0$, i.e. $\mathbb{P}- a.s.$ for almost all $n\in\mathbb{N}$ we have
$\lVert Z_n\rVert_H\leq \sqrt{\lambda^{-1}\alpha \log{n}}$. In other words
\begin{align*}
 C:=\sup_{n\ge 0} \frac{\lVert Z_n \rVert_H}{\sqrt{\log{n}}}< \infty\text{ }\mathbb{P}-a.s.
\end{align*}
\end{proof}

Using Lemma \ref{lem:estimate for Gaussian RV} and the characterization of
$Q$-Wiener processes of Lemma \ref{lem:characterisation of Wienerprocess},
we now obtain its Schauder decomposition which can be seen as a Gaussian
version of Lemma \ref{lem:hoelderschauder}.

\begin{prop}\label{lem:representation of Wienerprocess}
Let $\alpha \in (0,1/2)$, let $(\phi_n)_{n\ge 0}$ be the Schauder functions and $(Z_n)_{n\ge 0}$ a
sequence of independent, $\mathcal{N}(0,Q)$-distributed Gaussian variables,
where $Q$ is a positive self adjoint trace class operator on $H$. The
series defined process
\begin{align*}
 W_t=\sum_{n=0}^{\infty}\phi_n(t)Z_n,\quad t\in[0,1],
\end{align*}
converges $\P$-a.s. with respect to the $\lVert\cdot \rVert_\alpha$-norm on $[0,1]$ and is an $H$-valued $Q$-Wiener
process.
\end{prop}
\begin{proof}
We have to show that the process defined by the series satisfies the
conditions given in Lemma \ref{lem:characterisation of Wienerprocess}. The
first and the two last conditions concerning the covariance structure and
Gaussianity of scalar products have standard verifications. Let us just
argue for absolute and $\lVert\cdot \rVert_\alpha$-convergence of the series, thus proving
H\"older-continuity of the trajectories.

Since $T^H_\alpha$ is an isomorphism and since any single term of the series is even Lipschitz-continuous, it suffices to show that
\begin{align*}
   \left(T^H_\alpha \left( \sum_{n=0}^m \phi_n Z_n\right): m \in \N\right)
\end{align*}
is a Cauchy sequence in $\mathcal{C}_0^H$. Let us first calculate the image of  term $N$ under $T^H_\alpha$. We have
\begin{align*}
   (T^H_\alpha \phi_n Z_n)_N = 1_{\{n = N\}} c_N(\alpha) Z_N.
\end{align*}
Therefore for $m_1, m_2\ge 0, m_1\le m_2$
\begin{align*}
   \sum_{n=m_1}^{m_2} (T^H_\alpha \phi_n Z_n)_N = 1_{\{m_1 \le N \le m_2\}}c_N(\alpha) Z_N = \left(T^H_\alpha\left( \sum_{n=m_1}^{m_2} \phi_n Z_n\right)\right)_N.
\end{align*}
So if we can prove that $c_N(\alpha) Z_N$ a.s. converges to $0$ in $H$ as $N \rightarrow \infty$, the proof is complete. But this follows immediately from Lemma \ref{lem:estimate for Gaussian RV}: $c_N(\alpha)$ decays exponentially fast, and $\lVert Z_N \rVert_H \le C \sqrt{\log N}$.
\end{proof}

In particular we showed that for $\alpha < 1/2$ $W$ a.s. takes its trajectories in
\begin{align*}
   C_\alpha^0([0,1]; H) := \left\{ F:[0,1] \rightarrow H, F(0) = 0,
   \lim_{\delta \rightarrow 0} \sup_{\substack{t \neq s, \\ |t-s|
   < \delta}} \frac{\lVert F(t) - F(s) \rVert_H}{|t-s|^\alpha} = 0 \right\}
\end{align*}
By Lipschitz continuity of the scalar product, we also have $\langle F, e_k \rangle \in C_\alpha^0 ([0,1];\R)$:

Since $P_k$
and $R_k$ are orthogonal projectors and therefore Lipschitz continuous, we
obtain that for $F \in C_\alpha^0([0,1];H)$
$$
   \sup_{k \ge 0} \lVert\langle F, e_k \rangle \rVert_\alpha \le \lVert F \rVert_\alpha.
$$
We also saw that $T^H_\alpha(W)$ is well defined almost surely. As a
special case this is also true for the real valued Brownian motion. We have
by Proposition \ref{lem:representation of Wienerprocess}
\begin{align*}
 T^H_\alpha(W)=(c_n(\alpha) Z_n)
\end{align*}
where $(Z_n)_{n\ge 0}$ is a sequence of i.i.d.
$\mathcal{N}(0,Q)-$variables.

Plainly, the representation of the preceding Lemma can be used to prove the
representation formula for $Q$-Wiener processes by scalar Brownian motions
according to \cite{DaPrato}, Theorem 4.3.
\begin{prop}\label{prop: series representation of Wienerprocess}
   Let $W$ be a $Q$-Wiener process. Then
   \begin{align*}
      W(t) = \sum_{k=0}^\infty \sqrt{\lambda_k} \beta_k (t) e_k\text{, }t\in [0,1],
   \end{align*}
   where the series on the right hand side $\mathbb{P}$-a.s. converges uniformly on $[0,1]$,
   and $(\beta_k)_{k\ge 0}$ is a sequence of independent real valued Brownian motions.
\end{prop}
\begin{proof}
Using arguments as in the proof of Theorem \ref{CiesilskiH} and Lemma
\ref{lem:estimate for Gaussian RV} to justify changes in the order of
summation we get
\begin{align*}
 W=\sum_{n=0}^{\infty}\phi_n Z_n=\sum_{k\ge 0}\sum_{n\ge 0}\phi_n \langle Z_n,e_k\rangle e_k=\sum_{k\ge 0}
 \sqrt{\lambda_k}\sum_{n\ge 0}\phi_n N_{n,k}e_k=\sum_{k\ge 0} \sqrt{\lambda_k}\beta_k e_k,
\end{align*}
where the equivalences are $\P$-a.s. and $(N_{n,k})_{n,k\ge 0},
(\beta_k)_{k\ge 0}$ are real valued iid $\mathcal{N}(0,1)$  random
variables resp. Brownian motions. For the last step we applied Proposition
\ref{lem:representation of Wienerprocess} for the one-dimensional case.
\end{proof}

\subsection{Large deviations}

Let us recall some basic notions of the theory of large deviations that
will suffice to prove the large deviation principle for Hilbert space
valued Wiener processes. We follow \citet{Dembo}. Let $X$ be a topological
Hausdorff space. Denote its Borel $\sigma$-algebra by $\B$.

\begin{defn}[Rate function]
   A function $I : X \rightarrow [0, \infty]$ is called a rate function if it is lower semi-continuous, i.e. if for every $C \ge 0$ the set
   \begin{align*}
       \Psi_I (C) := \{ x \in X: I(x) \le C \}
   \end{align*}
   is closed. It is called a good rate function, if $\Psi_I(C)$ is compact. For $A \in \B$ we define
   $I(A):= \inf_{x \in A} I(x)$.
\end{defn}

\begin{defn}[Large deviation principle]
   Let $I$ be a rate function. A family of probability measures $(\mu_\varepsilon)_{\varepsilon > 0}$ on $(X, \B)$
   is said to satisfy the large deviation principle (LDP) with rate function $I$ if for any closed
   set $F \subset X$ and any open set $G \subset X$ we have
   \begin{align*}
      \limsup_{\varepsilon \rightarrow 0} \varepsilon \log \mu_\varepsilon (F)& \le - I(F) \mbox{ and} \\
      \liminf_{\varepsilon \rightarrow 0} \varepsilon \log \mu_\varepsilon (G) & \ge -
      I(G).
   \end{align*}
\end{defn}

\begin{defn}[Exponential tightness]
   A family of probability measures $(\mu_\varepsilon)_{\varepsilon > 0}$ is said to be exponentially tight
   if for every $a>0$ there exists a compact set $K_a \subset X$ such that
   \begin{align*}
      \limsup_{\varepsilon \rightarrow 0} \varepsilon \log \mu_\varepsilon (K_a^c) <
      -a.
   \end{align*}
\end{defn}

In our approach to Schilder's Theorem for Hilbert space valued Wiener
processes we shall mainly use the following proposition which basically
states that the rate function has to be known for elements of a sub-basis
of the topology.

\begin{prop}\label{prop:exponential und ldp fuer offene
baelle, dann ldp}
   Let $\mathcal{G}_0$ be a collection of open sets in the topology of $X$ such that for every
   open set $ G \subset X$ and for every $x \in G$ there exists $G_0 \in \mathcal{G}_0$ such that
   $x \in G_0 \subset G$. Let $I$ be a rate function and let $(\mu_\varepsilon)_{\varepsilon > 0}$ be an
   exponentially tight family of probability measures. Assume that for every $G \in \mathcal{G}_0$ we have
   \begin{align*}
      - \inf_{x\in G} I(x) = \lim_{\varepsilon \rightarrow 0}\varepsilon \log \mu_\varepsilon (G).
   \end{align*}
   Then $I$ is a good rate function, and $(\mu_\varepsilon)_\varepsilon$ satisfies an LDP
   with rate function $I$.
\end{prop}

\begin{proof}
Let us first establish the lower bound. In fact, let $G$ be an open set. Choose $x\in G$, and a basis set $G_0$ such that $x\in G_0\subset G.$
Then evidently
$$\liminf_{\varepsilon\to 0} \varepsilon\ln \mu_\varepsilon(G) \ge \liminf_{\varepsilon\to 0} \varepsilon\ln \mu_\varepsilon(G_0) = - \inf_{y\in G_0} I(y) \ge -I(x).$$
Now the lower bound follows readily by taking the $\sup$ of $-I(x), x\in G,$
on the right hand side, the left hand side not depending on $x$.

For the upper bound, fix a compact subset $K$ of $X.$ For $\delta > 0$ denote
$$I^\delta(x) = (I(x)-\delta)\wedge \frac{1}{\delta},\quad x\in X.$$
For any $x\in K$, use the lower semicontinuity of $I$, more
precisely that $\{y\in X: I(y)> I^\delta(x)\}$ is open to choose a
set $G_x\in \mathcal{G}_0$ such that
$$- I^\delta(x) \ge \limsup_{\varepsilon\to 0} \varepsilon \ln \mu_\varepsilon(G_x).$$ Use compactness of $K$ to extract from the open cover
$K \subset \cup_{x\in K} G_x$ a finite subcover $K\subset \cup_{i=1}^n G_{x_i}.$
Then with a standard argument we obtain
$$\limsup_{\varepsilon\to 0} \varepsilon \ln \mu_\varepsilon(K) \le \max_{1\le i\le n} \limsup_{\varepsilon\to 0} \varepsilon \ln \mu_\varepsilon(G_{x_i})\le - \min_{1\le i\le n} I^\delta(x_i)
\le - \inf_{x\in K} I^\delta(x).$$
Now let $\delta\to 0.$ Finally use exponential tightness to show that $I$ is a good rate function (see \cite{Dembo}, Section 4.1).
\end{proof}

The following propositions show how large deviation principles are
transferred between different topologies on a space, or via continuous maps
to other topological spaces.

\begin{prop}[Contraction principle]\label{prop:contraction principle}
   Let $X$ and $Y$ be topological Hausdorff spaces, and let $I: X \rightarrow [0, \infty]$ be a good rate function. Let $f: X \rightarrow Y$ be a continuous mapping. Then
   \begin{align*}
      I' : Y \rightarrow [0, \infty], I'(y) = \inf\{ I(x): f(x) = y \}
   \end{align*}
   is a good rate function, and if $(\mu_\varepsilon)_{\varepsilon > 0}$
   satisfies an LDP with rate function $I$ on $X$, then $(\mu_\varepsilon \circ f^{-1})_{\varepsilon > 0}$
   satisfies an LDP with rate function $I'$ on $Y$.
\end{prop}

\begin{prop}\label{prop:ldp von grober auf feine topo}
   Let $(\mu_\varepsilon)_{\varepsilon > 0}$ be an exponentially tight family of probability
   measures on $(X, \B_{\tau_2})$ where $\B_{\tau_2}$ are the Borel sets of $\tau_2$.
   Assume $(\mu_\varepsilon)$ satisfies an LDP with rate function $I$ with respect to some
   Hausdorff topology $\tau_1$ on $X$ which is coarser than $\tau_2$, i.e. $\tau_2 \subset \tau_1$.
   Then $(\mu_\varepsilon)_{\varepsilon > 0}$ satisfies the LDP with respect to $\tau_2$, with good rate
   function $I$.
\end{prop}

The main idea of our sequence space approach to Schilder's Theorem for
Hilbert space valued Wiener processes will just extend the following large
deviation principle for a standard normal variable with values in $\R$ to
sequences of i.i.d. variables of this kind.

\begin{prop} \label{prop:ldp standard normal}
   Let $Z$ be a standard normal variable with values in $\R$,
   \begin{align*}
      I: \R \rightarrow [0,\infty), \, x \mapsto \frac{x^2}{2},
   \end{align*}
   and for Borel sets $B$ in $\R$ let $\mu_\varepsilon(B):= \P(\sqrt{\varepsilon} Z \in B)$.
   Then $(\mu_\varepsilon)_{\varepsilon > 0}$ satisfies a LDP with good rate function $I$.
\end{prop}

\section{Large Deviations for Hilbert Space Valued Wiener Processes}\label{LDP}

Ciesielski's isomorphism and the Schauder representation of Brownian motion
yield a very elegant and simple method of proving large deviation
principles for the Brownian motion. This was first noticed by \citet{Baldi}
who gave an alternative proof of Schilder's theorem based on this
isomorphism. We follow their approach and extend it to Wiener processes
with values on Hilbert spaces. In this entire section we always assume $0 <
\alpha < 1/2$. By further decomposing the orthogonal 1-dimensional Brownian
motions in the representation of an $H$-valued Wiener process by its
Fourier coefficients with respect to the Schauder functions, we describe it
by double sequences of real-valued normal variables.

\subsection{Appropriate norms}

We work with new norms on the spaces of $\alpha$-H\"older continuous
functions given by
\begin{align*}
   &\lVert F\rVert^{'}_\alpha := \lVert T^H_\alpha F\rVert_\infty = \sup_{k,n} \left| c_n(\alpha) \int_{[0,1]} \chi_n(s)
   d\langle  F, e_k \rangle(s) \right|, F \in C_\alpha^0([0,1];H), \\
   & \lVert f\rVert^{'}_\alpha := \lVert T_\alpha f\rVert_\infty = \sup_{n} \left| c_n(\alpha) \int_{[0,1]} \chi_n(s)df(s)
   \right|, f \in C_\alpha^0([0,1];\R).
\end{align*}
Since $T^{H}_\alpha$ is one-to-one, $\lVert.\rVert^{'}_\alpha$ is indeed a norm.
Also, we have $\lVert.\rVert_\alpha^{'} \le \lVert.\rVert_\alpha$. Hence the topology
generated by $\lVert.\rVert_\alpha^{'}$
is coarser than the usual topology on $C_\alpha^0([0,1],H)$.\\

Balls with respect to the new norms $U_\alpha^\delta(F) := \{ G \in
C_\alpha^0([0,1];H): \lVert G - F\rVert_\alpha^{'} < \delta\}$ for $F\in C_\alpha^0([0,1];H), \delta>0$, have a
simpler form for our reasoning, since the condition that for $\delta > 0$ a
function $G\in C^0_\alpha([0,1],H)$ lies in $U_\alpha^\delta(F)$ translates into
the countable set of one-dimensional conditions $|\langle
T^H_\alpha(F)_n-T^H_\alpha(G)_n, e_k\rangle|<\delta$ for all $n,k\ge 0.$
This will facilitate the proof of the LDP for the basis of open balls of
the topology generated by $\lVert.\rVert_\alpha^{'}$. We will first prove the LDP
in the topologies generated by these norms and then transfer the result to
the finer sequence space topologies using Proposition \ref{prop:ldp von
grober auf feine topo}, and finally to the original function space using
Ciesielski's isomorphism and Proposition \ref{prop:contraction principle}.

\subsection{The rate function}
Recall that $Q$ is supposed to be a positive self-adjoint trace-class
operator on H. Let $H_0 :=( Q^{1/2} H, \lVert \cdot \rVert_0)$, equipped
with the inner product
\begin{align*}
   \langle  x,y \rangle_{H_0} := \langle  Q^{-1/2} x, Q^{-1/2} y \rangle_H,
\end{align*}
that induces the norm $\lVert \cdot\rVert_0$ on $H_0$. We define the
Cameron-Martin space of the $Q$-Wiener process $W$ by
\begin{align*}
   \H:= \left\{ F \in C([0,1];H): F(\cdot) = \int_0^\cdot U(s) ds \mbox{ for some } U \in L^2([0,1]; H_0)
   \right\}.
\end{align*}
Here $L^2([0,1];H_0)$ is the space of measurable functions $U$ from $[0,1]$
to $H_0$ such that $\int_0^1 \lVert U\rVert^2_{H_0} dx < \infty$. Define the
function $I$ via
\begin{align*}
    &I: C([0,1];H) \rightarrow [0, \infty] \\
    &F\mapsto \inf \left\{ \frac{1}{2}\int_0^1 \lVert U(s)\rVert^2_{H_0} ds: U \in L^2([0,1];H_0), F(\cdot) =
    \int_0^\cdot U(s) ds \right\}
\end{align*}
where by convention $\inf \emptyset = \infty$. In the following we will
denote any restriction of $I$ to a subspace of $C([0,1];H)$ (e.g. to
$(C_\alpha([0,1];H)$) by $I$ as well. We will use the structure of $H$ to
simplify our problem. It allows us to compute the rate function $I$ from
the rate function of the one dimensional Brownian by the following Lemma.
\begin{lem}\label{lem:gleichheit rate functions}

   Let $\tilde{I}: C([0,1];\R)$ be the rate function of the Brownian motion, i.e.
\begin{align*}
   \tilde{I}(f) :=  \left\{\begin{array}{ll} \int_0^1 |\dot{f}(s)|^2 ds, & f (\cdot) = \int_0^{\cdot} \dot{f}(s) ds\,\,\mbox{for a square integrable function}\,\, \dot{f},\\  \infty, & \mbox{otherwise}.\end{array}\right.
\end{align*}
   Let $(\lambda_k)_{k\ge 0}$ be the sequence of eigenvalues of $Q$. Then for all $F \in C([0,1];H)$ we have
   \begin{align*}
      I(F) = \sum_{k=0}^\infty \frac{1}{\lambda_k} \tilde{I}(\langle  F, e_k \rangle).
   \end{align*}
   where we convene that $c/0 = \infty$ for $c > 0$ and $0/0 = 0$.
\end{lem}
\begin{proof}
Let $F\in C([0,1];H)$.
   \begin{enumerate}
      \item First assume $I(F) < \infty$. Then there exists $U \in L^2([0,1]; H_0)$ such
      that $F = \int_0^\cdot U(s) ds$ and thus $\langle F,e_k\rangle=\int_0^\cdot \langle U(s), e_k
      \rangle ds$ for $k\ge 0$.
      Consequently we have by monotone convergence
\begin{align*}
  \frac{1}{2}\int_0^1 \lVert U(s)\rVert^2_{H_0} ds &= \frac{1}{2} \int_0^1 \left\lVert\sum_{k=0}^\infty  \langle  U(s), e_k \rangle e_k\right\rVert_{H_0} ds\\
&=\frac{1}{2} \int_0^1 \sum_{k=0}^\infty \langle  U(s), e_k \rangle^2 \langle Q^{-\frac{1}{2}} e_k, Q^{-\frac{1}{2}}e_k\rangle ds\\
&=\frac{1}{2} \int_0^1 \sum_{k=0}^\infty  \frac{1}{\lambda_k} \langle  U(s), e_k \rangle^2 ds\\
& = \sum_{k=0}^\infty \frac{1}{\lambda_k}\tilde{I}(\langle  F, e_k
\rangle).
\end{align*}

The last expression does not depend on the choice of $U$. Hence we get that
$I(F) < \infty$ implies $I(F)=\sum_{k=0}^\infty \frac{1}{\lambda_k}
\tilde{I}(\langle  F, e_k \rangle)$.

\item Conversely assume $\sum_{k=0}^\infty \frac{1}{\lambda_k} \tilde{I}(\langle  F, e_k \rangle) < \infty$.
Since $\tilde{I}(\langle  F, e_k \rangle) < \infty$ for all $k\ge 0$, we
know that there exists a sequence $(U_k)_{k\ge 0}$ of square-integrable
real-valued functions such that $\langle  F, e_k \rangle = \int_0^\cdot
U_k(s) ds$. Further, those functions $U_k$ satisfy by monotone convergence
\begin{align*}
  \int_0^1\sum_{k=0}^\infty \frac{1}{\lambda_k} |U_k(s)|^2 ds = \sum_{k=0}^\infty \frac{1}{\lambda_k}
  \int_0^1 |U_k(s)|^2 ds = \sum_{k=0}^\infty \frac{2}{\lambda_k} \tilde{I}(\langle  F, e_k \rangle)<
  \infty.
\end{align*}
 So if we define $U(s):= \sum_{k=0}^\infty U_k(s) e_k, s\in[0,1]$, then $U \in L^2([0,1];H_0)$. This follows from
\begin{align*}
 U\in  L^2([0,1];H_0) &\text{ iff }
 \int_0^1 \lVert U(s)\rVert^2_{H_0} ds = \int_0^1\sum_{k=0}^\infty \frac{1}{\lambda_k} |U_k(s)|^2 ds< \infty.
\end{align*}
Finally we obtain by dominated convergence ($\lVert F(t)\rVert_H<\infty$)
\begin{align*}
 F(t) = \sum_{k=0}^\infty \langle  F(t), e_k \rangle e_k = \sum_{k=0}^\infty e_k \int_0^t U_k(s) ds =
 \int_0^t U(s) ds,
\end{align*}
such that
\begin{align*}
 I(F) \le \frac{1}{2}\int_0^1 \lVert U(s)\rVert^2_{H_0} ds = \frac{1}{2}\int_0^1 \sum_{k=0}^\infty
 \frac{1}{\lambda_k} |U_k(s)|^2 ds < \infty.
\end{align*}
\end{enumerate}

Combining the two steps we obtain $I(F)<\infty$ iff $\sum_{k=0}^\infty
\frac{1}{\lambda_k} \tilde{I}(\langle  F, e_k \rangle) < \infty$ and in
this case
\begin{align*}
 I(F) = \sum_{k=0}^\infty \frac{1}{\lambda_k} \tilde{I}(\langle  F, e_k \rangle).
\end{align*}
This completes the proof.
\end{proof}
Lemma \ref{lem:gleichheit rate functions} allows us to show that $I$ is a rate function.
\begin{lem}\label{lem: rate function}
   $I$ is a rate function on $(C_\alpha^0([0,1];H), \lVert.\lVert_\alpha^{'})$.
\end{lem}

\begin{proof}
    For a constant $C\ge 0$ we have to prove that if $(F_n)_{n\ge 0} \subset \Psi_I(C) \cap C_\alpha^0([0,1];H)$ converges
   in $C_\alpha^0([0,1];H)$ to $F$, then $F$ is also in $\Psi_I(C)$.

   It was observed in \citet{Baldi} that $\tilde{I}$ is a rate function for the $\lVert . \rVert_\alpha^{'}$-topology
   on $C_0^\alpha([0,1|;\R)$. By our assumption we know that for every $k\in\mathbb{N}$, $(\langle  F_n, e_k
   \rangle)_{n\ge 0}$ converges in $(C_0^\alpha([0,1|;\R),\lVert .\rVert_\alpha^{'})$ to
   $\langle  F, e_k \rangle$. Therefore
   \begin{align*}
      \tilde{I} (\langle  F, e_k \rangle) \le \liminf_{n \rightarrow \infty} \tilde{I}(\langle  F_n,
      e_k \rangle),
   \end{align*}
   so by Lemma \ref{lem:gleichheit rate functions} and by Fatou's lemma
   \begin{align*}
      C &\ge \liminf_{n \rightarrow \infty} I(F_n) =  \liminf_{n \rightarrow \infty} \sum_{k=0}^\infty
      \frac{1}{\lambda_k} \tilde{I}(\langle  F_n, e_k \rangle) \ge \sum_{k=0}^\infty  \frac{1}{\lambda_k}
      \liminf_{n \rightarrow \infty} \tilde{I}(\langle  F_n, e_k \rangle) \\
      & \ge \sum_{k=0}^\infty \frac{1}{\lambda_k} \tilde{I} (\langle  F, e_k\rangle) = I(F).
   \end{align*}
   Hence $F \in \Psi_I (C)$.
\end{proof}

\subsection{LDP for a sub-basis of the coarse topology}
To show that the $Q$-Wiener process $(W(t): t \in [0,1])$ satisfies a LDP
on $(C_\alpha([0,1];H), \lVert .\rVert_\alpha)$ with good rate function $I$ as
defined in the last section we now show that the LDP holds for open balls
in our coarse topology induced
by $\lVert.\rVert_\alpha^{'}$. The proof is an extension of the version of \citet{Baldi} for the real valued Wiener process.

For $\varepsilon>0$ denote by $\mu_\varepsilon$ the law of $\sqrt{\varepsilon} W$, i.e.
$\mu_\varepsilon(A) = \P( \sqrt{\varepsilon}W \in A)$,
$A\in\B(H)$.

\begin{lem}\label{lem: ldp for basis}
   For every $\delta > 0$ and every $F \in C_\alpha^0([0,1]; H)$ we have
   \begin{align*}
      \lim_{\varepsilon \rightarrow 0} \varepsilon \log \mu_\varepsilon (U^\delta_\alpha(F))
      = - \inf_{G \in U^\delta_\alpha(F)} I(G).
   \end{align*}
\end{lem}

\begin{proof}
1. Write $T^H_\alpha F = (\sum_{k=0}^\infty F_{n,k}e_k)_{n\in\N}$. Then $\sqrt{\varepsilon}W$ is
in $U^\delta_\alpha(F)$ if and only if
   \begin{align*}
      \sup_{k,n \ge 0} \left| \sqrt{\varepsilon} c_n(\alpha) \int_0^1 \chi_n d \langle  W,
      e_k\rangle -F_{k,n} \right| < \delta.
   \end{align*}
Now for $k\ge 0$ we recall $\langle  W, e_k \rangle = \sqrt{\lambda_k} \beta_k$, where $(\beta_k)_{k\ge 0}$ is a sequence of independent standard
Brownian motions. Therefore for $n,k\ge 0$
   \begin{align*}
      \left| \int_0^1 \chi_n d \langle  W, e_k\rangle \right| = \left| \sqrt{\lambda_k} Z_{k,n} \right| ,
   \end{align*}
   where $(Z_{k,n})_{k,n\ge 0}$ is a double sequence of independent standard normal variables. Therefore by independence
   \begin{align*}
      \mu_\varepsilon (U^\delta_\alpha(F)) & =  \P\left(\bigcap_{k,n \in \N_0}
      \left| c_n(\alpha) \sqrt{\varepsilon\lambda_k} Z_{k,n} -F_{k,n} \right| <\delta \right) \\
      & = \prod_{k=0}^\infty \prod_{n=0}^\infty \P\left( c_n(\alpha)
      \sqrt{\varepsilon\lambda_k} Z_{k,n} \in (F_{k,n} - \delta , F_{k,n} + \delta ) \right).
   \end{align*}
   To abbreviate, we introduce the notation
   \begin{align*}
       \P_{k,n} (\varepsilon) =  \P\left( c_n(\alpha) \sqrt{\varepsilon\lambda_k} Z_{k,n} \in
       (F_{k,n} - \delta , F_{k,n} + \delta ) \right)\text{, }\varepsilon >0, n,k\in\mathbb{N}_0.
   \end{align*}
   For every $k\ge 0$ we split $\N_0$ into subsets $\Lambda^k_i,$ $i=1,2,3,4$, for each of which we will
   calculate $\prod_{k =0}^\infty\prod_{n \in \Lambda^k_i} \P_{n,k} (\varepsilon)$ separately. Let
   \begin{align*}
      & \Lambda^k_1 = \{ n \ge 0: 0 \notin [ F_{k,n} - \delta , F_{k,n} + \delta ]\} \\
      & \Lambda^k_2 = \{ n \ge 0: F_{k,n} = \pm \delta \} \\
      & \Lambda^k_3 = \{ n \ge 0: [- \delta/2, \delta/2] \subset [ F_{k,n} - \delta ,
      F_{k,n} + \delta  ]\} \\
      & \Lambda^k_4 = ( \Lambda^k_1 \cup \Lambda^k_2\cup\Lambda^k_3)^c.
   \end{align*}
   By applying Ciesielski's isomorphism to the real-valued functions $\langle F, e_k\rangle$, we see that for every fixed $k$, $\Lambda_3^k$
   contains nearly all $n$. Since $(T_\alpha^H F)_n$ converges to zero in $H$, in particular $\sup_{k \ge 0} |F_{k,n}|$ converges to zero as $n \rightarrow \infty$. But for every fixed $n$, $(F_{k,n})_k$ is in $l^2$ and therefore converges to zero. This shows that for large enough $k$ we must have $\Lambda_3^k = \N_0$, and therefore $\cup_{k}(\Lambda^k_3)^c$ is finite.

2. First we examine $\prod_{k=0}^\infty \prod_{n \in \Lambda^k_3} \P_{k,n} (\varepsilon)$.
Note that for $n \in \Lambda_3^k$ we have
   \begin{align*}
       [- \delta/2, \delta/2] \subset [ F_{k,n} - \delta  , F_{k,n} + \delta   ],
   \end{align*}
   and therefore
   \begin{align*}
      \prod_{k=0}^\infty \prod_{n \in \Lambda^k_3} \P_{k,n}(\varepsilon) & \ge \prod_{k=0}^\infty
      \prod_{n \in \Lambda^k_3} \P\left(Z_{k,n} \in \left(-\frac{\delta}{ 2c_n(\alpha)
      \sqrt{\varepsilon\lambda_k} }, \frac{ \delta} {2c_n(\alpha) \sqrt{\varepsilon\lambda_k} }
      \right) \right) \\
      & = \prod_{k=0}^\infty  \prod_{n \in \Lambda^k_3} \left( 1 - \sqrt{\frac{2}{\pi}}
      \int_{\delta/(2c_n(\alpha)\sqrt{\varepsilon\lambda_k})}^\infty e^{-u^2/2} du \right).
   \end{align*}
   For $a > 1$ we have $\int_a^\infty e^{-x^2/2} dx \le e^{-a^2/2}$. Thus for small enough
   $\varepsilon$:
   \begin{align*}
      \prod_{k=0}^\infty \prod_{n \in \Lambda^k_3} \P_{k,n}(\varepsilon) \ge
      \prod_{k=0}^\infty \prod_{n \in \Lambda^k_3} \left( 1 - \sqrt{\frac{2}{\pi}}
      \exp\left( - \frac{\delta^2}{8 c_n^2(\alpha) \varepsilon \lambda_k} \right)\right).
   \end{align*}
   This amount will tend to $1$ if and only if its logarithm tends to 0 as $\varepsilon \rightarrow 0$. Since $\log(1-x) \leq - x$ for $x\in (0,1)$, it suffices to prove that
   \begin{align} \label{eq:lambda3}
      \lim_{\varepsilon \rightarrow 0} \sum_{k=0}^\infty \sum_{n \ge 0}
      \exp\left( - \frac{\delta ^2}{8 c_n^2(\alpha) \varepsilon \lambda_k} \right) = 0.
   \end{align}
This is true by dominated convergence, because $c_n(\alpha) = 2^{n(\alpha-1/2) + \alpha - 1}$,
and since $(\lambda_k) \in l_1$.

We will make this more precise.
First observe that for $a>0$
\begin{align*}
 e^{-a}\leq \frac{1}{a}e^{-1}\\
\mbox{if }\log(a)-a\leq -1.
\end{align*}
For $k,n\ge 0$ we write $\eta_{n,k}=\frac{\delta^2}{8 c_n^2(\alpha) \varepsilon
\lambda_k}$. Clearly there exists a finite set $T\subset \mathbb{N}_0^2$
such that $\log(\eta_{n,k})-\eta_{n,k}\leq -1$ for all $(n,k)\in T^c$. We
set $C=\sum_{(n,k)\in T}e^{-\eta_{n,k}}$ and get
\begin{align*}
 \sum_{k=0}^\infty \sum_{n= 0}^\infty  \exp\left( - \frac{\delta^2}{8 c_n^2(\alpha)
 \varepsilon \lambda_k} \right)&=C+\sum_{(n,k)\in T^c}^\infty e^{-\eta_{n,k}}\\
&\leq C+\sum_{(n,k)\in T^c}^\infty\frac{1}{\eta_{n,k}}e^{-1}\\
&\leq C+\frac{8\varepsilon e^{-1}}{\delta^2}\sum_{k\ge 0}
\lambda_k\sum_{n\ge 0}c_n(\alpha)^2 <\infty.
\end{align*}

3. Since $\cup_{k\ge 0}\Lambda_4^k$ is finite, and since for every $n$ in $\Lambda_4^k$
the interval $(\F_{k,n} - \delta, \F_{k,n} + \delta )$ contains a small
neighborhood of $0$, we have
   \begin{align} \label{eq:lambda4}
     \lim_{\varepsilon \rightarrow 0}\prod_{k=0}^\infty
     \prod_{n \in \Lambda^k_4} \P_{k,n}(\varepsilon) = 1.
   \end{align}

4. Again because $\cup_{k\ge 0}\Lambda_2^k$ is finite, we obtain from its definition that
   \begin{align}\label{eq:lambda2}
       \lim_{\varepsilon \rightarrow 0}\prod_{k=0}^\infty \prod_{n \in
       \Lambda^k_2}  \P_{k,n}(\varepsilon) = 2^{-|\cup_k \Lambda_2^k|}.
   \end{align}

5. Finally we calculate $\lim_{\varepsilon \rightarrow 0} \prod_{k=0}^\infty
\prod_{n \in \Lambda^k_1}  \P_{k,n}(\varepsilon)$. For given $k,n$ define
   \begin{align*}
      \bar{F}_{k,n} = \left\{\begin{array}{ll} F_{k,n} - \delta, &F_{k,n} > \delta, \\ F_{k,n} + \delta, &
      F_{k,n} < - \delta. \end{array} \right.
   \end{align*}
   We know that $Z_{k,n}$ is standard normal, so that by Proposition
   \ref{prop:ldp standard normal} for $n \in \Lambda_1^k$
   \begin{align*}
      \lim_{\varepsilon \rightarrow 0} \varepsilon \log \P_{k,n}^0(\varepsilon)
      = - \frac{\bar{F}_{k,n}^2}{2c_n^2(\alpha) \lambda_k},
   \end{align*}
   and therefore again by the finiteness of $\cup_k \Lambda_1^k$
   \begin{align} \label{eq:lambda1}
      \lim_{\varepsilon \rightarrow 0} \varepsilon  \log \prod_{k=0}^\infty \prod_{n \in \Lambda_1^k}
      \P_{k,n}^0(\varepsilon) = - \sum_{k=0}^\infty \sum_{n \in \Lambda_1^k}
      \frac{\bar{F}_{k,n}^2}{2c_n^2(\alpha) \lambda_k}.
   \end{align}

6. Combining (\ref{eq:lambda3}) - (\ref{eq:lambda1}) we obtain
   \begin{align*}
     \lim_{\varepsilon \rightarrow 0} \varepsilon \log \mu_\varepsilon (U^\delta_\alpha(F)) =  -
     \sum_{k=0}^\infty \frac{1}{\lambda_k} \sum_{n \in \Lambda_1^k} \frac{\bar{F}_{k,n}^2}{2c_n^2(\alpha)}.
   \end{align*}
   So if we manage to show
   \begin{align*}
      - \sum_{k=0}^\infty \frac{1}{\lambda_k} \sum_{n \in \Lambda_1^k}
      \frac{\bar{F}_{k,n}^2}{2c_n^2(\alpha)} = -\inf_{G \in U^\delta_\alpha(F)} I(G),
   \end{align*}
   the proof is complete. By Ciesielski's isomorphism, every $G \in C_\alpha^0([0,1];H)$ has the
   representation
   \begin{align*}
      G = \sum_{k=0}^\infty e_k \sum_{n=0}^\infty \frac{G_{k,n}}{c_n(\alpha)} \phi_n.
   \end{align*}
Its derivative fulfills (if it exists) for any $k\ge 0$
   \begin{align*}
      \langle  \dot G, e_k \rangle = \sum_{n=0}^\infty \frac{G_{k,n}}{c_n(\alpha)} \chi_n.
   \end{align*}
  Since the Haar functions $(\chi_n)_{n\ge 0}$ are a CONS for $L^2([0,1])$, we see that $\tilde{I}(\langle  G,
  e_k \rangle) < \infty$ if and only if $(G_{k,n}/ c_n(\alpha)) \in l_2$, and in this case
   \begin{align*}
      \tilde{I} (\langle  G, e_k \rangle) = \frac{1}{2} \int_0^1 \langle  \dot G(s), e_k \rangle^2 ds =
      \sum_{n=0}^\infty \frac{G_{k,n}^2}{2c_n^2(\alpha)}.
   \end{align*}
   So we finally obtain with Lemma \ref{lem:gleichheit rate functions} the desired equality
   \begin{align*}
     \inf_{G \in U^\delta_\alpha(F)} I(G) &= \inf_{G \in U^\delta_\alpha(F)}
      \sum_{k=0}^\infty \frac{1}{\lambda_k} \tilde{I}(\langle  G, e_k \rangle) =
      \inf_{G \in U^\delta_\alpha(F)} \sum_{k=0}^\infty \frac{1}{\lambda_k}  \sum_{n=0}^\infty
      \frac{G_{k,n}^2}{2c_n^2(\alpha)} \\
      & = \sum_{k=0}^\infty \frac{1}{\lambda_k} \sum_{n \in \Lambda_1^k}
      \frac{\bar{F}_{k,n}^2}{2c_n^2(\alpha)}.
   \end{align*}
\end{proof}

\subsection{Exponential tightness}

The final ingredient needed in the proof of the LDP for Hilbert space valued Wiener processes is exponential tightness.
It will be established in two steps. The first step claims exponential tightness for the family of laws of $\sqrt{\varepsilon}Z, \varepsilon>0,$ where $Z$ is an $H$-valued $\mathcal{N}(0,Q)$-variable.
\begin{lem}\label{lem: exponential tight elementary}
Let $\varepsilon>0$ and $\nu_\varepsilon=\P\circ (\sqrt{\varepsilon}Z)^{-1}$ for a centered Gaussian random variable $Z$ with values in the separable Hilbert space $H$ and covariance operator $Q$. Then $(\nu_\varepsilon)_{\varepsilon\in (0,1]}$ is exponentially tight. More precisely for every $a>0$ there exists a compact subset $K_a$ of $H$, such that for every $\varepsilon \in (0,1]$
\begin{align*}
   \nu_\varepsilon(K_a^c) \le e^{-a/\varepsilon}
\end{align*}

\end{lem}
\begin{proof}
We know that for a sequence $(b_k)_{k\ge 0}$ converging to 0, the operator $T_{(b_k)}:=\sum_{k=0}^\infty b_k \langle \cdot, e_k\rangle e_k$ is compact.  That is, for bounded sets $A\subset H$ the set $T_{(b_k)}(A)$ is precompact in $H$. Since $H$ is complete, this means that $cl(T_{(b_k)}(A))$ is compact. Let $a' > 0$ to be specified later. Denote by $B(0,\sqrt{a'})\subset H$ the ball of radius $\sqrt{a'}$ in $H$. We will show that there exists a zero sequence $(b_k)_{k\ge 0}$, such that the compact set $K_{a'}=cl(T_{(b_k)}(B(0,\sqrt{a'})))$ satisfies for all $\varepsilon \in (0,1]$
\begin{align}\label{eq: estimate for compact(a')}
   \P(\sqrt{\varepsilon}Z\in (K_{a'})^c)\leq c e^{-a'/\varepsilon}.
\end{align}
with a constant $c>0$ that does not depend on $a'$. Thus for given $a$, we can choose $a'>a$ such that for every $\varepsilon \in (0,1]$
\begin{align*}
   c \le e^{(a'-a)/\varepsilon}
\end{align*}
and therefore the proof is complete once we proved \eqref{eq: estimate for compact(a')}.

Since $Z$ is Gaussian, $e^{\lambda \lVert Z\rVert_H}$ is integrable for small $\lambda$, and we can apply Markov's inequality to obtain constants $\lambda(Q),c(Q) > 0$ such that $\P(\lVert Z\rVert_H\ge \sqrt{a'})\le c(Q)e^{-\lambda(Q) a'}$.

Note that if $(\lambda_k)_{k\ge 0}\in l^1$, we can always find a sequence $(c_k)_{k\ge 0}$ such that $\lim_{k\to\infty}c_k = \infty$ and $\sum_{k\ge 0} c_k\lambda_k<\infty$. For $\beta>0$ that will be specified later, we set $b_k=\sqrt{\frac{\beta}{c_k}}$ for all $k\ge 0$. We can define $(T_{(b_k)})^{-1}=\sum_{k=0}^\infty \frac{1}{b_k}\langle \cdot, e_k\rangle e_k$. This gives
\begin{align*}
 \P(\sqrt{\varepsilon}Z\in (K_{a'})^c)&\le\P(\sqrt{\varepsilon}(T_{(b_k)})^{-1}(Z)\notin B(0,\sqrt{a'}))\\
&=\P(\lVert (T_{(b_k)})^{-1}(Z)\rVert_H^2 \ge \frac{a'}{\varepsilon})\\
&=\P\left(\sum_{k=0}^\infty c_k |\langle Z,e_k\rangle|^2\ge  \frac{\beta a'}{\varepsilon}\right)\\
&=\P\left(\lVert \tilde{Z}\rVert_H\ge \sqrt{\frac{\beta a'}{\varepsilon}}\right),
\end{align*}
where $\tilde{Z}$ is a centered Gaussian random variable with trace class covariance operator
\begin{align*}
   \tilde{Q}=\sum_{k=0}^\infty c_k\lambda_k\langle \cdot, e_k\rangle e_k.
\end{align*}
Consequently we obtain
\begin{align*}
   \P(\sqrt{\varepsilon}Z\in (K_{a'})^c)\le c(\tilde{Q})e^{-\frac{\lambda(\tilde{Q})\beta a'}{\varepsilon}}
\end{align*}
Choosing $\beta=\frac{1}{\lambda(\tilde{Q})}$ proves the claim \eqref{eq: estimate for compact(a')}.
\end{proof}

With the help of Lemma \ref{lem: exponential tight elementary} we are now in a position to prove exponential tightness for the family $(\mu_\varepsilon)_{\varepsilon \in (0,1]}$.
\begin{lem}\label{lem:exponential tightness}
   $(\mu_\varepsilon)_{\varepsilon \in (0,1]}$ is an exponentially tight family of probability measures on $(C_\alpha^0([0,1];H)$, $\lVert.\rVert_\alpha)$.
\end{lem}

\begin{proof}
Let $a > 0$. We will construct a suitable set of the form
\begin{align*}
   \tilde{K}^a = \prod_{n=0}^\infty K^a_n
\end{align*}
such that
\begin{align*}
    \limsup_{\varepsilon \rightarrow 0}\varepsilon \log\mu_{\varepsilon}\left[\left(\left(T^H_\alpha\right)^{-1}\tilde{K}^a\right)^c\right]\leq -a.
\end{align*}
Here each $K^a_n$ is a compact subset of $H$, such that the diameter of $K^a_n$ tends to 0 as $n$ tends to $\infty$. Then $\tilde{K}^a$ will be sequentially compact in $\mathcal{C}^H_0$ by a diagonal sequence argument. Since $\mathcal{C}^H_0$ is a metric space, $\tilde{K}^a$ will be compact. As we saw in Theorem \ref{CiesilskiH}, $(T^H_{\alpha})^{-1}$ is continuous, so that then $K^a:=(T^H_{\alpha})^{-1}(\tilde{K}^a)$ is compact in $(C^0_{\alpha}([0,1],H), \lVert \cdot \rVert_{\alpha})$.

Let $\nu_\varepsilon=\P\circ (\sqrt{\varepsilon}Z)^{-1}$ for a random variable $Z$ on $H$ with $Z\sim \mathcal{N}(0,Q)$. By Lemma \ref{lem: exponential tight elementary}, we can find a sequence of compact sets $(K^a_n)_{n\in\mathbb{N}}\subset H$ such that for all $\varepsilon \in (0,1]$:
\begin{align*}
 \nu_\varepsilon((K^a_n)^c)\leq \exp\left(\frac{-(n+1)a}{\varepsilon}\right).
\end{align*}

To guarantee that the diameter of the $K^a_n$ converges to zero, denoting by $\overline{B}(0,d)$ the closed ball of radius $d$ around $0$, we set
\begin{align*}
   \tilde{K^a}:=\prod_{n=0}^{\infty}c_n(\alpha)\left(\overline{B}\left(0,\sqrt{\frac{a(n+1)}{\lambda}}\right)\cap K^a_n\right).
\end{align*}
Since $c_n(\alpha)\sqrt{a(n+1)/\lambda}\rightarrow0$ as $n\rightarrow \infty$, this is a compact set in $\mathcal{C}^H_0$. Thus $K^a:=(T^H_{\alpha})^{-1}(\tilde{K^a})$ is compact in $(C^0_{\alpha}([0,1],H), \lVert \cdot \rVert_{\alpha})$.

  Remember that by Lemma \ref{lem:representation of Wienerprocess} we have $W=\sum_{n=0}^{\infty}\phi_nZ_n$, where $(Z_n)_{n\ge 0}$ is an i.i.d. sequence of $\mathcal{N}(0,Q)-$ variables. This implies $T_\alpha^H(W)=(c_n(\alpha)Z_n)_{n\ge 0}$ and thus for any $\varepsilon \in (0,1]$
\begin{align*}
  \mu_{\varepsilon}((K^a)^c)&=\P\left[\cup_{n\in\mathbb{N}_0}\left\{c_n(\alpha)\sqrt{\varepsilon}Z_n\in \left(c_n(\alpha)\left(B\left(0,\sqrt{\frac{a(n+1)}{\lambda}}\right)\cap K^a_n\right)\right)^c\right\}\right]\\
  &\leq \sum_{n=0}^{\infty}\left(\nu_\varepsilon((K^a_n)^c)+\P\left(\lVert Z_n\rVert\geq \sqrt{\frac{a(n+1)}{\varepsilon\lambda}}\right)\right)\\
  &\leq \sum_{n=0}^{\infty}\left(e^{\frac{-(n+1)a}{\varepsilon}}+ce^{\frac{-a(n+1)}{\varepsilon}}\right)\\
  &=(1+c)\frac{e^{\frac{-a}{\varepsilon}}}{1-e^{\frac{-a}{\varepsilon}}}.
\end{align*}
  So we have
\begin{align*}
 \limsup_{\varepsilon \rightarrow 0}\varepsilon \log\mu_{\varepsilon}((K^a)^c)\leq -a,
\end{align*}
\end{proof}

We now combine the arguments given so far to obtain an LDP in the H\"older spaces.
\begin{lem}
   $(\mu_\varepsilon)_{\varepsilon \in (0,1]}$ satisfies an LDP on $(C_\alpha^0([0,1];H), \lVert.\rVert_\alpha)$ with good rate function $I$.
\end{lem}
\begin{proof}
   We know $\lVert.\rVert_\alpha^{'} \le \lVert.\rVert_\alpha$. Therefore the $\lVert.\rVert_\alpha^{'}$-topology is coarser, which in turn implies that every compact set in the $\lVert.\rVert_\alpha$-topology is also a compact set in the $\lVert.\rVert_\alpha^{'}$-topology. From Lemma \ref{lem:exponential tightness} we thus obtain that $(\mu_\varepsilon)_{\varepsilon \in (0,1]}$ is also exponentially tight on $(C_\alpha^0([0,1];H), \lVert.\rVert_\alpha^{0})$.

   Proposition \ref{prop:exponential und ldp fuer offene baelle, dann ldp} implies that $(\mu_\varepsilon)_{\varepsilon \in (0,1]}$ satisfies an LDP with good rate function $I$ on $(C_\alpha^0([0,1];H)$, $\lVert.\rVert_\alpha^{0})$.

   Finally we obtain from Proposition \ref{prop:ldp von grober auf feine topo} and from Lemma \ref{lem:exponential tightness} that $(\mu_\varepsilon)_{\varepsilon \in (0,1]}$ satisfies an LDP with good rate function $I$ on $(C_\alpha^0([0,1];H), \lVert.\rVert_\alpha)$.
\end{proof}

We may now extend the LDP from $(C_\alpha^0([0,1];H), \lVert.\rVert_\alpha)$ to $(C_\alpha([0,1];H), \lVert.\rVert_\alpha)$. This is an immediate consequence of the contraction principle (Proposition \ref{prop:contraction principle}), since the inclusion map from $C_\alpha^0([0,1];H)$ to $C_\alpha([0,1];H)$ is continuous.
Similarly we can transfer the LDP from $C^0_\alpha([0,1];H)$ to $C([0,1]; H)$, the space of continuous functions on $[0,1]$ with values in $H$, equipped with the uniform norm.

\begin{thm}
   Let $(W(t): t \in [0,1])$ be a $Q$-Wiener process and let for $\varepsilon \in (0,1]$ $\mu_\varepsilon$ be the law of $\sqrt{\varepsilon} W$. Then $(\mu_\varepsilon)_{\varepsilon \in (0,1]}$ satisfies an LDP on $(C([0,1];H),\lVert.\rVert_\infty)$ with rate function $I$.
\end{thm}
\begin{proof}
   First we can transfer the LDP from $(C^0_\alpha([0,1];H), \lVert.\rVert_\alpha)$ to $(C^0_\alpha([0,1];H)$, $\lVert.\rVert_\infty)$. This is because on $C^0_\alpha([0,1];H)$, $\lVert.\lVert_\infty \le \lVert.\rVert_\alpha$, whence the $\lVert.\rVert_\infty$-topology is coarser. Therefore $I$ is a good rate function for the $\lVert.\rVert_\infty$-topology as well, and $(\mu_\varepsilon)_{\varepsilon \in (0,1]}$ satisfies an LDP on $(C^0_\alpha([0,1];H)$, $\lVert.\rVert_\infty)$ with good rate function $I$.

   The inclusion map from $(C^0_\alpha([0,1];H)$, $\lVert.\rVert_\infty)$ to $(C([0,1];H)$, $\lVert.\rVert_\infty)$ is continuous, so that an application of the contraction principle (Proposition \ref{prop:contraction principle}) finishes the proof.
\end{proof}

\vspace{20pt}

\noindent \textbf{Acknowledgement:} Nicolas Perkowski is supported by a Ph.D. scholarship of the Berlin Mathematical School.

\begin{footnotesize}
\bibliography{literatur}
\bibliographystyle{plainnat}
\end{footnotesize}

\end{document}